\documentclass[preprint]{elsarticle}
\usepackage{lineno,hyperref}
\modulolinenumbers[5]

\journal{ArXiv}

\usepackage{amsmath}
\usepackage{amsthm}
\usepackage{amsfonts}
\usepackage{enumerate}
\usepackage[utf8]{inputenc}
\usepackage[T1]{fontenc}
\usepackage{amssymb}
\usepackage{color,graphics,srcltx}
\usepackage[mathscr]{eucal}
\usepackage{dsfont}
\usepackage{accents}
\begin{document}

\theoremstyle{plain}
\newtheorem{theorem}{Theorem}[section]
\newtheorem{corollary}[theorem]{Corollary}
\newtheorem{lemma}[theorem]{Lemma}
\newtheorem{proposition}[theorem]{Proposition}
\theoremstyle{definition}
\newtheorem{definition}{Definition}
\newtheorem{example}{Example}
\newtheorem{problem}{Problem}
\theoremstyle{remark}
\newtheorem{remark}{Remark}

\def\la{{\langle}}
\def\ra{{\rangle}}
\def\IC{{\mathbb{C}}}
\def\IR{{\mathbb{R}}}
\def\LL{{\mathbb{L}}}
\def\HH{{\mathbb{H}}}
\def\DD{{\mathbb{D}}}
\def\TT{{\mathbb{T}}}
\def\cL{{\cal L}}
\def\cP{{\cal P}}
\def\cB{{\cal B}}
\def\cH{{\cal H}}
\def\cK{{\cal K}}
\def\cC{{\cal C}}
\def\cU{{|cal U}}
\def\cF{{\cal F}}
\def\cF{{\mathcal F}}
\def\cP{{\mathcal P}}
\def\cR{{\mathcal R}}
\def\cS{{\mathcal S}}
\def\cT{{\mathcal T}}
\def\cV{{\mathcal V}}
\def\BH{{\cB(H)}}
\def\bA{{\bf A}}
\def\bB{{\bf B}}
\def\bC{{\bf C}}
\def\bD{{\bf D}}
\def\bF{{\bf F}}
\def\bG{{\bf G}}
\def\bQ{{\bf Q}}
\def\bE{{\bf E}}
\def\bK{{\bf K}}
\def\bM{{\bf M}}
\def\bY{{\bf Y}}
\def\bb{{\bf b}}
\def\ba{{\bf a}}
\def\bc{{\bf c}}
\def\bx{{\bf x}}
\def\bu{{\bf u}}
\def\bL{{\bf L}}
\def\bX{{\bf X}}
\def\b0{{\bf 0}}
\def\tr{{\rm tr}}
\def\Ker{{\rm Ker}}
\def\){{\right)}}
\def\({{\left(}}
\def\Ra{{\ \Rightarrow\ }}
\def\Lra{{\ \Leftrightarrow\ }}
\def\[{\left [}
\def\]{\right ]}
\def\Re{{\rm Re\,}}
\def\diag{{\rm diag}}
\def\rank{{\rm rank}}
\def\span{{\rm span}}
\def\co{{\rm conv}}
\def\){ \right)}
\def\({\left(}
\def\qed{\hfill\vbox{\hrule width 6 pt
\hbox{\vrule height 6 pt width 6 pt}} \medskip}
\def\vec{{\rm vec}}
\def\conv{{\rm con}}
\def\cl{{\bf cl}}

\begin{frontmatter}
\title{Convexity and Star-shapedness of Matricial Range}

\author[polyu]{Pan-Shun Lau}
\ead{panshun.lau@polyu.edu.hk}
\address[polyu]{Department of Applied Mathematics, The Hong Kong Polytechnic University, Hong Kong.}

\author[CWM]{Chi-Kwong Li}
\ead{ckli@math.wm.edu}
\address[CWM]{Department of Mathematics, College of William \& Mary, Williamsburg, VA 23185.}

\author[ISU]{Yiu-Tung Poon}
\ead{ytpoon@iastate.edu}
\address[ISU]{Department of Mathematics, Iowa State University, Ames, IA 50011.\\[20pt] {\bf\normalsize Dedicated to Professor Yik-Hoi Au-Yeung}}

\author[polyu]{Nung-Sing Sze}
\ead{raymond.sze@polyu.edu.hk}



\begin{abstract}
Let ${\bf A} = (A_1, \dots, A_m)$ be an $m$-tuple of bounded linear operators acting on a Hilbert space $\mathcal H$. Their joint $(p,q)$-matricial range $\Lambda_{p,q}({\bf A})$ is the collection of $(B_1, \dots, B_m) \in {\bf M}_q^m$, where $I_p\otimes B_j$ is a compression of $A_j$ on a $pq$-dimensional subspace. This definition covers various kinds of generalized numerical ranges for different values of $p,q,m$. In this paper, it is shown that $\Lambda_{p,q}({\bf A})$ is star-shaped if the dimension of $\mathcal H$ is sufficiently large. If $\dim {\mathcal H}$ is infinite, we extend the definition of $\Lambda_{p,q}({\bf A})$ to $\Lambda_{\infty,q}({\bf A})$ consisting of $(B_1, \dots, B_m) \in {\bf M}_q^m$ such that $I_\infty \otimes B_j$ is a compression of $A_j$ on a closed subspace of ${\mathcal H}$, and consider the joint essential $(p,q)$-matricial range $$\Lambda^{ess}_{p,q}({\bf A}) = \bigcap \{ {\bf cl}(\Lambda_{p,q}(A_1+F_1, \dots, A_m+F_m)): F_1, \dots, F_m \hbox{ are compact operators}\}.$$ Both sets are shown to be convex, and the latter one is always non-empty and compact.
 \end{abstract}


\end{frontmatter}

\section{Introduction}
Let $\cB(\cH)$ be the algebra of bounded linear operators acting on a complex
Hilbert space $\cH$. If $\cH$ has dimension $n < \infty$, we identify $\cB(\cH)$
with $\bM_n$, the space of $n \times n$ complex matrices.
The numerical range of $A \in \cB(\cH)$ is defined and denoted by
$$W(A) = \{\langle Ax, x \rangle: x \in \cH,~ \|x\| = 1\}.$$
It is a useful concept for studying matrices and operators; see~\cite{HJ,GR}.
The Toeplitz-Hausdorff Theorem asserts that this set is always
convex \cite{Hau,Toe}, i.e. $tw_1+(1-t)w_2\in W(A)$ for all $w_1,\ w_2\in W(A)$ and $0\le t\le 1$.
As shown by many researchers, there are interesting interplay between
the geometrical properties of the numerical ranges and the
algebraic and analytic properties of the operators;
for example; see \cite{AL,Halmos,HJ,GR}.
Motivated by problems from theoretical and applied areas, researchers have
considered different generalizations of the numerical range,
and extended the results on the classical numerical range to the
generalized numerical ranges. We mention a few of them related to our study in the following.

Let $\cV_q$ denote the set of operators $X: \cK \rightarrow \cH$
for some $q$-dimensional subspace $\cK$ of $\cH$ such that $X^*X = I_{\cK}$.
To study the compressions of $A \in \cB(\cH)$ on a subspace of $\cH$,
researchers consider the  $q$-matricial range defined by
$$W(q: A) = \{ X^*A X: X\in \cV_{q}\} \subseteq \bM_q.$$
One may see the basic references \cite{LT,Thompson,TW} and
the excellent survey \cite{Farenick} on the topic. We remark that
$W(q: A)$ is called spatial matricial range in \cite{Farenick}.

In the study of joint
behavior of several operators in $\cB(\cH)$, researchers
consider the joint numerical range of an $m$-tuple
$\bA = (A_1, \dots, A_m) \in \cB(\cH)^m$,
$$W(\bA) = \left\{\left( \langle A_1x,x\rangle, \dots, \langle A_mx, x\rangle\right):
x\in \cH, ~\|x\| = 1\right\}.$$
In the study of control theory, this is known as the $m$-multiform
numerical range, and the convexity of the sets is useful;
see \cite{AYP,Fan,LP0} and their references.

In connection to the study of quantum error correction, researchers study
the $(p,q)$-matricial range $\Lambda_{p,q}(A)$ of $A \in \cB(\cH)$ defined as follows.
Let $p,q$ be positive integers with $pq \le \dim \cH$. Then
$$\Lambda_{p,q}(A) = \{ B \in \bM_q:  X^*AX = I_p \otimes B \hbox{ for some }
X \in \cV_{pq}\}.$$
When $q = 1$, the definition reduces to the rank $p$-numerical range of $A$
defined by
$$\Lambda_p(A) = \{ b: X^*AX = b I_p \hbox{ for some } X \in \cV_p\}.$$
One may see \cite{CKZ,LP2,LPS} and their references for the background of these
concepts.\footnote{Note that in \cite{LPS}, the definition
of $\Lambda_{p,q}(A)$  is slightly different from but equivalent to ours.}

In fact, in the study of quantum error correction, it is more important to study
the joint $(p,q)$-matricial range and the joint rank $p$-numerical range of
an $m$-tuple of operators $\bA = (A_1, \dots, A_m)$ defined, respectively, by
$$\Lambda_{p,q}(\bA) 
= \left\{(B_1,\dots,B_m)\in \bM_q^m:
X^*A_jX = I_p \otimes B_j\hbox{ with }
X \in \cV_{pq} \hbox{ for } j = 1,\dots, m\right\},$$
and
$$\Lambda_{p}(\bA) 
= \left\{(b_1,\dots,b_m): X^*A_jX = b_jI_p \hbox{ with }
X \in \cV_{p} \hbox{ for } j = 1,\dots, m\right\}.$$
Of course, one may also consider the special case when $p = 1$, and define the
joint $q$-matricial range of $\bA$ by
$$W(q:\bA) = \{ (X^*A_1X, \dots, X^*A_mX): X \in \cV_q\}.$$

We are interested in the geometrical properties of the generalized
numerical ranges mentioned above.
In \cite{LPS}, it was shown that the
$\Lambda_{p,q}(A)$ could be quite delicate even for one Hermitian
matrix $A$.  In \cite{LP0}, it was shown that
the joint numerical range $W(A_1, \dots, A_m)$
may not be convex if $m \ge 4$; moreover, if $\{I,A_1,A_2,A_3\}$ is linearly
independent, then one can always find a rank-2 orthogonal projection
$A_4$ such that $W(A_1,A_2,A_3,A_4)$
is not convex. When the generalized numerical range fails to be convex,
researchers try to establish some weaker and useful
geometric properties. Let $V$ be a vector space
over $\IR$ or $\IC$. A subset $S$ of $V$ is said to be {\it star-shaped}, if there exists $s_0\in S$
such that  $ts_0+(1-t)s\in S$ for all $s\in S$ and $0\le t\le 1$. The point $s_0$ is called a
{\it star-center} of $S$.
Some  star-shapedness and convexity results on $\Lambda_{p}(\bA)$
were obtained in \cite{LP1,LP2} provided that the underlying Hilbert space has a high dimension.

In Section 2, we will show that $\Lambda_{p,q}(\bA)$ is star-shaped if $\dim \cH$ is
sufficiently large. If $\cH$ is infinite dimensional, the dimension condition holds
automatically.
As a result, the sets $\Lambda_{p,q}(\bA)$, $W(q:A)$ are always star-shaped;
the images of these sets under affine maps are all star-shaped.
As we shall see, this will further imply the star-shapedness of
other generalized numerical ranges.

If $\dim\cH$ is infinite, we extend the definition of
$\Lambda_{p,q}(A_1, \dots, A_m)$ to
$\Lambda_{\infty,q}(A_1, \dots, A_m)$ consisting of $(B_1, \dots, B_m) \in \bM_q^m$ such
that
$$I_\infty \otimes B_j = B_j \oplus B_j \oplus B_j \oplus \cdots$$
is a compression of $A_j$
on an infinite dimensional closed subspace $\cK$ of $\cH$.
In Section 3, we show that  the set $\Lambda_{\infty,q}(\bA)$ is always convex.

In connection to the study of
operators in the Calkin algebra, we consider the
joint essential $(p,q)$-matricial range of $\bA$ defined by
$$\Lambda^{ess}_{p,q}(\bA) =
\bigcap\{ \cl\left(\Lambda_{p,q}(A_1 + F_1, \dots, A_m + F_m)\right):\ F_1, \dots, F_m\in \cB(\cH)
 \hbox{ are compact operators} \},$$
where $\cl(X)$ denotes the closure of the set $X$.
In Section 4, we show that $\Lambda_{p,q}^{ess}(\bA)$ is the same as
the joint essential $q$-matricial range
$$
W_{ess}(q:\bA) =
\bigcap\{ \cl\left(W(q:(A_1 + F_1, \dots, A_m + F_m))\right) :\ F_1, \dots, F_m\in \cB(\cH)
 \hbox{ are compact operators} \}
$$
considered by other researchers;
see \cite{Arveson,LP2,Muller,Paulsen} and their references.
Moreover, we show that $\Lambda_{p,q}^{ess}(\bA)
= W_{ess}(q:\bA)$ is always a non-empty compact
convex set. As a result, the sets $\Lambda_{p,q}^{ess}(\bA)$ are the same for all positive integers $p$.
Furthermore, in the definitions of
the sets $W_{ess}(q:\bA)$ and
$\Lambda_{p,q}^{ess}(\bA)$, we show that one can replace
$F_1, \dots, F_m$ by finite rank operators.
These extend the results in \cite{LP1,LP2}.
Some related results and problems will be discussed in Section 5.

\medskip
To conclude this section, we mention some reductions that can be used in our
study.
Firstly, let $\cS(\cH)$ be the real linear space of self-adjoint operators in $\cB(\cH)$
and identify $\cS(\cH)$ with ${\bf H}_n$, the space of $n \times n$ Hermitian matrices
when $\dim \cH = n < \infty$.
Every $A \in \cB(\cH)$ can be written as $A = H+iG$ for
a pair of $H, G\in \cS(\cH)$. The set $\Lambda_{p,q}(A)$
can be identified with
$$\Lambda_{p,q}(H,G) = \left\{ (B,C) \in
{\bf H}_q^2: X^*HX = I_p \otimes B ~\hbox{ and }~  X^*G X = I_p \otimes C,\ X\in \cV_{pq}\right\}.$$
Therefore, one may focus on the joint $(p,q)$-matricial range of $m$ self-adjoint operators,
i.e., $A_1,\dots, A_m \in \cS(\cH)$.

\medskip
Secondly, suppose $T = (t_{ij}) \in \bM_m(\IR)$ is nonsingular,
and $B_j = \sum_{i=1}^m t_{ij} A_{i}$.  Then
$(Z_1, \dots, Z_m) \in \Lambda_{p,q}(B_1, \dots, B_m)$
if and only if $Z_j = \sum_{i=1}^m t_{ij} Y_i$ for some
$(Y_1, \dots, Y_m) \in \Lambda_{p,q}(A_1, \dots, A_m)$.
Also, we may  assume that $\{A_1, \dots, A_m\}$ is linearly independent.

\medskip
Finally, we state some standard properties of $\Lambda_{p,q}(\bA)$ that will be used in this paper.
\begin{enumerate}
\item  The joint $(p,q)$-matricial range is invariant under simultaneous unitary conjugation, i.e.,
$$\Lambda_{p,q}(U^* A_1U, \dots, U^*A_mU) = \Lambda_{p,q} (A_1, \dots, A_m)
\quad \hbox{for all unitary $U\in \cB(\cH)$.}$$
\item Suppose $(B_1,\dots, B_m)$ is a compression of $(A_1,\dots, A_m)$. Then
$$\Lambda_{p,q}(B_1,\dots, B_m) \subseteq \Lambda_{p,q}(A_1,\dots, A_m).$$
\item Suppose $(X_1,\dots,X_m) \in \Lambda_{p,q}(A_1,\dots, A_m)$ and 
$(Y_1,\dots,Y_m)  \in \Lambda_{p,q}(B_1,\dots, B_m)$. Then
$$\left(tX_1 + (1-t)Y_1,\dots, tX_m + (1-t) Y_m \right) \in \Lambda_{p,q}\left( A_1 \oplus B_1, \dots, A_m \oplus B_m\right) \quad\hbox{for all } t \in [0,1].$$
\end{enumerate}

\section{Star-shapedness}

In this section, we show that $\Lambda_{p,q}(\bA)$ is always star-shaped if the
dimension of $\cH$ is sufficiently large. We also give some estimations on
the dimension of $\cH$ that ensure the star-shapedness, and non-emptyness of
$\Lambda_{p,q}(\bA)$. Some consequences of the results will be mentioned.

For $1\le r \le \dim \cH$, let $\cV_r^\bot$
be the set of operators $X: \cK^\bot \to \cH$ such that $X^*X = I_{\cK^\bot}$, 
where $\cK^\bot$ is the orthogonal complement of a $r$-dimensional subspace $\cK$ of $\cH$.

\begin{theorem} \label{main0}
 Let $\bA = (A_1, \dots, A_m)\in \cS(\cH)^m$ be an $m$-tuple of self-adjoint operators in $\cB(\cH)$.
 Suppose
$$\bC = (C_1, \dots, C_m) \in \Lambda_{p,q}(Y^*A_1Y, \dots, Y^*A_mY) \quad \hbox{for any} \quad Y \in \cV_{pq(m+1)}^\bot.$$
Then $\bC$ is a star-center of $\Lambda_{p,q}(\bA)$.
\end{theorem}

\begin{proof}
Suppose $\bB = (B_1,\dots, B_m) \in \Lambda_{p,q}(\bA)$.
Then there is $X_1:\cK_1 \to \cH$ with $X_1^*X_1 = I_{\cK_1}$
for some $pq$-dimensional subspace $\cK_1$ of $\cH$
such that $X_1^*A_j X_1 = I_p \otimes B_j$ for $j =1 ,\dots,m$.
Extend the operator $X_1$ to an unitary operator $U:\cH \to \cH$ so that $U|_{\cK_1} = X_1$.
Let $\cL$ be the subspace spanned by
$$\left\{\cK_1,\,  U^*A_1U(\cK_1),\, \dots,\, U^*A_mU(\cK_1) \right\}.$$
Then $\cL$ has dimension at most $pq(m+1)$.
By extending the subspace $\cL$ if necessary, we may assume $\dim \cL = pq(m+1)$.
Define $Y = U|_{\cL^\bot}$.
Then $Y^*Y = I_{\cL^\bot}$ and $Y \in \cV^\bot_{pq(m+1)}$.
Furthermore, as $X_1  = U|_{\cK_1}$, $Y = U|_{\cL^\bot}$
and $\cK_1  \subseteq \cL$, we have $Y^*X_1 = 0$.
Also for any $u \in \cL^\bot$ and $v \in \cK_1$,
$U^*A_jUv \in \cL$ and hence
$$\langle u, Y^*A_j X_1 v\rangle
= \langle  Yu, A_j X_1 v \rangle
= \langle  Uu, A_j U v \rangle
= \langle  u, U^*A_j U v \rangle
= 0$$
and thus, $Y^*A_j X_1 = 0$ for all $j =1,\dots,m$.
Now by the assumption,
$$\bC = (C_1, \dots, C_m) \in \Lambda_{p,q}(Y^*A_1Y, \dots, Y^*A_mY).$$
There exists $X_2: \cK_2 \to  \cL^\bot$ with $X_2^*X_2 = I_{\cK_2}$
for some $pq$-dimensional subspace $\cK_2$ of $\cL^\bot$
such that $X_2^*(Y^* A_j Y) X_2 = I_p \otimes C_j$ for $j = 1,\dots,m$.
Observe that  $\cK_1$ and $\cK_2$ are two distinct $pq$-dimensional subspaces of $\cH$ and are orthogonal to each other.
Furthermore, $X_1: \cK_1 \to \cH$ and $YX_2: \cK_2 \to \cH$ satisfy 
$$X_1^*X_1 = I_{\cK_1},\quad (YX_2)^*(YX_2) = I_{\cK_2},\quad 
(YX_2)^*X_1 = 0
\quad\hbox{and}\quad
(YX_2)^*A_j X_1 = 0 \quad \hbox{for } j =1,\dots,m.$$
Then the operator $Z = X_1 \oplus (YX_2): \cK_1 \oplus \cK_2 \to \cH$
satisfies $V^*V = I_{\cK_1\oplus \cK_2}$ and 
$$Z^*A_j Z =
\begin{bmatrix}
I_p \otimes B_j & 0  \cr
0 & I_p \otimes C_j
\end{bmatrix}\quad j=1,\dots,m,$$
with respect to the $2pq$-dimensional subspace $\cK_1 \oplus \cK_2$.
Clearly, 
$\bB \in \Lambda_{p,q}(I_p \otimes B_1,\dots, I_p \otimes B_m)$ and $\bC \in \Lambda_{p,q}(I_p\otimes C_1,\dots, I_p \otimes C_m)$.
By the two properties as stated at the end of Section 1, for any $t\in [0,1]$, 
$$t \bB + (1-t)\bC = \left(t B_1 + (1-t) C_1,\dots, t B_m + (1-t) C_m \right) \in \Lambda_{p,q} \left(Z^* A_1 Z, \dots, Z^*A_m Z  \right)
\subseteq \Lambda_{p,q}(A_1,\dots, A_m);$$
hence $\bC$ is a star-center of $\Lambda_{p,q}(\bA)$.
\end{proof}

By Theorem~\ref{main0}, if $\bigcap \{\Lambda_{p,q}(Y^*A_1Y,\dots,Y^*A_1Y):Y\in\cV^\bot_{pq(m+1)}\} $ is non-empty, then $\Lambda_{p,q}(\bA)$ is star-shaped and $\bigcap \{\Lambda_{p,q}(Y^*A_1Y,\dots,Y^*A_1Y):Y\in\cV^\bot_{pq(m+1)}\} $ is a subset of the star-center of $\Lambda_{p,q}(\bA)$.
\vskip .1in
In particular, we have the following result for the joint $q$-matricial range.

\begin{corollary} \label{main}
 Let $\bA = (A_1, \dots, A_m)\in \cS(\cH)^m$ be an $m$-tuple of self-adjoint operators in $\cB(\cH)$.
Suppose
$$\bC = (C_1, \dots, C_m) \in W(q:(Y^*A_1Y, \dots, Y^*A_mY)) \quad \hbox{for any} \quad Y \in \cV_{q(m+1)}^\bot.$$
Then $\bC$ is a star-center of $W(q:\bA)$.
\end{corollary}

In general, it may not be easy to check whether one can find
$(C_1, \dots, C_m) \in {\mathbf H}_q^{m}$ satisfying the assumption in
Theorem \ref{main0} and Corollary \ref{main}.  In this connection, we have
the following.

\begin{theorem} \label{main2}
Let $\bA = (A_1, \dots, A_m)\in \cS(\cH)^m$ be an $m$-tuple of self-adjoint operators in $\cB(\cH)$.
Suppose $\dim \cH \ge (pq(m+2)-1)(m+1)^2$. Then
\begin{itemize}
\item[{\rm (1)}]
$\Lambda_{pq(m+2)}(\bA)$ is non-empty.
\item[{\rm (2)}]
$\Lambda_{p,q}(\bA)$ is star-shaped and
$(c_1 I_q, \dots, c_m I_q)$ is a star-center of $\Lambda_{p,q}(\bA)$
for all $(c_1,\dots,c_m) \! \in \!\Lambda_{pq(m+2)}(\bA)$.
\item[{\rm (3)}] For every real affine map $L: {\bf H}_{q}^m \rightarrow \IR^r$, the set
 $$\{L(B_1, \dots, B_m): (B_1, \dots, B_m)\in \Lambda_{p,q}(\bA)\}$$
 is star-shaped with star-center $L(c_1I_q, \dots, c_mI_q)$ for all $(c_1,\dots,c_m) \in \Lambda_{pq(m+2)}(\bA)$.
\end{itemize}
\end{theorem}

\medskip
Below, we give several examples of real affine maps in (3).

\medskip
\begin{enumerate}
\item Let $L(B_1, \dots, B_m) = ( L_1(B_1), \dots, L_m(B_m))$,
for any affine functions $L_1, \dots, L_m$.

\item   Let $L(B_1, \dots, B_m) = (\tr B_1, \dots, \tr B_m)$.
We get the joint $q$-numerical range in the Halmos-Berger sense when $p=1$;
see \cite{Halmos,LP0}.

\item  Let $L(B_1, \dots, B_m) = (\tr (CB_1), \dots, \tr (C B_m))$ for a
matrix $C\in {\bf H}_q$. We get the joint $C$-numerical range when $p=1$;
see \cite{AYT,Chien2014,CLP}.
\end{enumerate}

To prove Theorem \ref{main2}, we
need the following results on $\Lambda_{p}(\bA)$ obtained in
\cite[Proposition 2.4 and Proposition 2.5]{LP2}.

\begin{proposition}\label{2.4}
Let $\bA = (A_1, \dots, A_m)\in \cS(\cH)^m$ be an $m$-tuple of self-adjoint operators in
$\cB(\cH)$ and $k>1$. If $\dim \cH \ge (k- 1)(m+1)^2$, then $\Lambda_{k}(\bA)$  is non-empty.
\end{proposition}

\begin{proposition} \label{2.5}
Let $\bA = (A_1, \dots, A_m)\in \cS(\cH)^m$ be an $m$-tuple of self-adjoint operators in
$\cB(\cH)$. Suppose $1 \le r < k \le \dim \cH$. Then
$$\Lambda_{k}(\bA) \subseteq \bigcap \left\{ \Lambda_{k-r} (Y^*\bA Y): Y \in \cV^\bot_r \right\}
\quad\hbox{where}\quad Y^*\bA Y = (Y^*A_1Y, \dots, Y^*A_m Y).$$
\end{proposition}

\noindent
\it Proof of Theorem \ref{main2}.
\rm
Note that  $(\lambda_1 I_{q},\dots,\lambda_m I_{q}) \in \Lambda_{p,q}(\bA)$
if and only if
$(\lambda_1,\dots,\lambda_m) \in \Lambda_{pq}(\bA)$.
Then (1) follows immediately from Proposition \ref{2.4}.

\medskip
(2) Suppose $(c_1,\dots,c_m) \in \Lambda_{pq(m+2)}(\bA)$. For all $ Y \in \cV^\bot_{pq(m+1)}$,   by Proposition \ref{2.5} with $k=pq(m+2)$ and $r=pq(m+1)$,  we have
$(c_1,\dots,c_m) \in \Lambda_{pq}(Y^*\bA Y)$.
Hence, $(c_1I_q, \dots, c_mI_q)\in \Lambda_{p,q}(Y^*\bA Y)$. Therefore, the result follows from Theorem \ref{main0}.

\medskip
Clearly, (3) follows from (2).
\qed

\medskip
Note that when $m=1,2$, the bound $(k- 1)(m+1)^2$ in
Proposition \ref{2.4} can be lower to $(m+1)k-m$. In these cases, the bound in Theorem \ref{main2}
can be lower to $(m + 1) (m + 2) pq - m$;
see {\rm \cite{CKZ,LPS0}}.

\medskip

In the following, we obtain additional results on $\Lambda_{p,q}(\bA)$ that are useful in
the analysis in the next section in addition to their own interest.


\vskip 10pt
\begin{proposition}\label{thm27}
Let $\bA = (A_1, \dots, A_m)\in \cS(\cH)^m$ be an $m$-tuple of self-adjoint operators in $\cB(\cH)$.
If $1\le qr <p\le\dim \cH $, then
$$\Lambda_{p,q}(\bA)\subseteq \bigcap \{\Lambda_{p-qr,q}(Y^*\bA Y): Y \in \cV^\bot_r\}
\quad\hbox{where}\quad Y^*\bA Y = (Y^*A_1Y, \dots, Y^*A_m Y).$$
\end{proposition}

\begin{proof}
We start with the case when $r = 1$, that is,
$$\Lambda_{p,q}(\bA)\subseteq \bigcap \{\Lambda_{p-q,q}(Y^*\bA Y): Y \in \cV^\bot_1\}.$$
Suppose $Y \in \cV_1^\bot$. Then there is an one-dimensional subspace $\cL = \mathrm{span}\{ y
\}$ of $\cH$ with  $y \in \cH$ such that $Y: \cL^\bot \to \cH$ with $Y^*Y = I_{\cL^\bot}$.
Extend the map $Y$ to an unitary operator $U: \cH \to \cH$ such that $U|_{\cL^\bot} = Y$.
Then $Y^*AY$ is a compression of $U^*AU$ on $\cL^\bot$.

Now let $\bC  = (C_1,...,C_m)\in \Lambda_{p,q}(\bA) = \Lambda_{p,q}(U^*\bA U)$.
Then there exists an operator $X:\cK_1\to \cH$ with $X^*X = I_{\cK_1}$ for
some $pq$-dimensional subspace $\cK_1$ of $\cH$
such that $X^*(U^*A_jU) X = I_p \otimes C_j$, $j = 1,\dots,m$.
With respect to the compression $I_p\otimes C_j$ on $\cK_1$,
let
$$\{e_{11},\dots, e_{1q},e_{21},\dots,e_{2q},\dots,e_{p1},\dots, e_{pq}\}$$
be the corresponding basis of $\cK_1$.
Define the $q \times p$ matrix
$$Q = \begin{bmatrix}
\langle y, Xe_{11}\rangle & \dots & \langle y, Xe_{1q}\rangle \cr
\vdots & \ddots & \vdots \cr
\langle y, Xe_{p1}\rangle & \dots & \langle y, Xe_{p q}\rangle
\end{bmatrix}.$$
Note that the nullity of $Q$ is at least $k = p-q$. Then there exists a $p \times k$ matrix
$W = \left[w_{ij} \right]$ with orthonormal columns such that $W^*W = I_k$ and $Q W = 0$.
Thus,
$$\left\langle y, X\sum_{k=1}^p w_{js}  e_{jt} \right\rangle  = \sum_{k=1}^p w_{js} \langle y, Xe_{jt} \rangle = 0 \quad \hbox{for $s = 1,\dots,p$ and $t = 1,\dots,k$.}$$
Now define a $pk$-dimensional subspace $\cK_2 =\span \{e_{11},\dots, e_{1k},e_{21},\dots,e_{2q},\dots,e_{p1},\dots, e_{pk}\}$ and set 
$X_2 = X(W \otimes I_q)$ being an operator from $\cK_2$ to $\cH$, equivalently,
$$X_2: e_{st} \mapsto X\sum_{k=1}^p w_{js}  e_{jt}\quad \hbox{for $s = 1,\dots,p$ and $t = 1,\dots,k$.}$$
Then $X^*_2 X_2 = I_{\cK_2}$ and 
from the above equalities, $X_2(\cK_2)$ is orthogonal to $\{y\}$ and hence $X_2(\cK_2) \subseteq \cL^\bot$.
Therefore,
$$X_2^*(Y^*A_j Y)X_2 =X_2 ^*(U^*A_jU)X_2 = (W^*\otimes I_q) X^* U^*A_j UX (W \otimes I_q)
=   (W^* \otimes I_q) (I_p \otimes C_j) (W \otimes I_q)
= I_k \otimes C_j$$
for all $ j =1,\dots,m$. Thus, $\bC \in \Lambda_{k,q}(Y^*\bA Y)=\Lambda_{p-q,q}(Y^*\bA Y)$ and the result follows for $r = 1$.
The general case follows by induction as 
$$\bigcap_{Y_1\in\cV_k^\bot} \Lambda_{p-qk,q}(Y_1^*\bA Y_1)\subseteq \bigcap_{Y_1\in\cV_k^\bot} \bigcap_{Y_2\in\cV_1^\bot} \Lambda_{p-qk-q,q}(Y_2^*Y_1^*\bA Y_1Y_2) \subseteq \bigcap_{Y\in\cV_{k+1}^\bot} \Lambda_{p-q(k+1),q}(Y^*\bA Y) .$$
\end{proof}

\begin{corollary}\label{coro25}
Let $\bA = (A_1, \dots, A_m)\in \cS(\cH)^m$ be an $m$-tuple of self-adjoint operators in $\cB(\cH)$.
Suppose $\dim \cH \ge \left(pq\left(q^2\left(m+1\right)+1\right)-1\right)(m+1)^2$. Then
\begin{itemize}
\item[{\rm (1)}]
$\Lambda_{\tilde p,q}(\bA)$ is non-empty where $\tilde p=p(q^2(m+1)+1)$.
\item[{\rm (2)}]
 $\Lambda_{p,q}(\bA)$ is star-shaped and
$(C_1 , \dots, C_m )$ is a star-center of $\Lambda_{p,q}(\bA)$
for all $(C_1 , \dots, C_m ) \in \Lambda_{\tilde p,q}(\bA)$.
\end{itemize}
\end{corollary}

\begin{proof}
Recall that $(\lambda_1 I_q,...,\lambda_m I_q)\in \Lambda_{\tilde p,q}(\bA)$ if and only if $(\lambda_1,...,\lambda_m)\in\Lambda_{\tilde pq}(\bA)$. Then (1) follows from Proposition~\ref{2.4}.

On the other hand, by Theorem~\ref{thm27}, we have
$$
\Lambda_{\tilde p,q}(\bA) \subseteq \bigcap \left\{ \Lambda_{\tilde p-pq^2(m+1),q}(Y^*\bA Y):Y\in\cV^\bot_{pq(m+1)} \right\}\\
= \bigcap \{ \Lambda_{p,q}(Y^*\bA Y): Y \in \cV^\bot_{pq(m+1)}\}.
$$
Then (2) follows by Theorem~\ref{main0}.
\end{proof}

\medskip
As discussed at the end of Section 1, $\Lambda_{p,q}(A_1,\dots,A_m)$ can be identified with $\Lambda_{p,q}(H_1,\dots,H_m, G_1,\dots, G_m)$,
where $A_j = H_j + i G_j$ for $H_j,G_j \in \cS(\cH)$. Therefore, most of the results in this section, including Theorem \ref{main0}, Corollary  \ref{main}, and Proposition \ref{thm27}, actually hold if one replaces $\bA = (A_1,\dots, A_m) \in \cS(\cH)^m$ by $\bA = (A_1,\dots, A_m) \in \cB(\cH)^m$,
while some minor modifications are required for Theorem \ref{main2} and Corollary \ref{coro25}, which are restated as follows.

\begin{theorem} \label{main5}
Let $\bA = (A_1, \dots, A_m)\in \cB(\cH)^m$ be an $m$-tuple of bounded linear operators.
\begin{itemize}
\item[{\rm (1)}] Set $p_1 = 2pq(m+1)$.
If $\dim \cH \ge (p_1-1)(2m+1)^2$, then $\Lambda_{p_1}(\bA)$ is non-empty
and $\Lambda_{p,q}(\bA)$ is star-shaped and
$(c_1 I_q, \dots, c_m I_q)$ is a star-center of $\Lambda_{p,q}(\bA)$
for all $(c_1,\dots,c_m) \! \in \!\Lambda_{p_1}(\bA)$.

\item[\rm (2)] Set $p_2 = p(q^2(2m+1)+1)$.
If $\dim \cH \ge \left(p_2q-1\right)(2m+1)^2$, then
$\Lambda_{p_2,q}(\bA)$ is non-empty 
and $\Lambda_{p,q}(\bA)$ is star-shaped and
$(C_1 , \dots, C_m )$ is a star-center of $\Lambda_{p,q}(\bA)$
for all $(C_1 , \dots, C_m ) \in \Lambda_{p_2,q}(\bA)$.
\end{itemize}
\end{theorem}

\section{Convexity of $(\infty,q)$-matricial range}

In this section, {\it we always assume that $\cH$ has infinite-dimension}.
We then extend the definition of  $\Lambda_{p,q}(A_1, \dots, A_m)$ to
$\Lambda_{\infty,q}(A_1, \dots, A_m)$ consisting of $(B_1, \dots, B_m) \in \bM_q^m$ such
that
$$I_\infty \otimes B_j = B_j \oplus B_j \oplus B_j \oplus \cdots$$
is a compression of $A_j$
on an infinite dimensional closed subspace $\cK$ of $\cH$.
We will show that $\Lambda_{\infty,q}(\bA)$ is always convex.
To prove this,  we need some related concepts and auxiliary results.
Denote by $\cF(\cH)$  the set of all finite rank operators in $\cB(\cH)$
and $\cS_\cF(\cH) = \cS(\cH) \cap \cF(\cH)$.
Let $\cV_{\infty}$ be the set of operator $X: \cK \to \cH$
such that $X^*X = I_{\cK}$
for an infinite dimensional subspace $\cK$ of $\cH$.
Also define $\cV^\bot = \bigcup_{r\ge 1} \cV^\bot_r.$

It is clear that $\Lambda_{\infty,q}(\bA) \subseteq \Lambda_{p,q}(\bA)$ for all $p\ge 1$.
The following result is a consequence of Corollary~\ref{coro25}.

\begin{proposition}\label{prop31}
Let $\bA = (A_1, \dots, A_m)\in \cS(\cH)^m$ be an $m$-tuple of self-adjoint operators in $\cB(\cH)$.
Then
$\Lambda_{p,q}(\bA)$ is always star-shaped for all positive integers $p,q$. Moreover,
if
$\bC \in \Lambda_{\infty,q}(\bA)$, then $\bC$ is a star-center of
$\Lambda_{p,q}(\bA)$ for any positive integer $p$.
\end{proposition}


\begin{theorem}\label{thm33}
Let $\bA = (A_1, \dots, A_m)\in \cS(\cH)^m$ be an $m$-tuple of self-adjoint operators in $\cB(\cH)$ 
and $p_0$ positive integer.
Denote by $S_{p,q}(\bA)$ the set of star-centers of $\Lambda_{p,q}(\bA)$. Then
$$
\Lambda_{\infty,q}(\bA)
= \bigcap_{p\ge 1} S_{p,q}(\bA)
= \bigcap_{p\ge 1} \Lambda_{p,q}(\bA)  
= \bigcap \{ \Lambda_{p_0,q}(Y^*\bA Y): Y \in \cV^\bot\} 
= \bigcap \{ \Lambda_{p_0,q}(\bA + \bF): \bF \in \cS_\cF(\cH)^m\}
$$
where  $Y^*\bA Y = (Y^*A_1Y, \dots, Y^*A_m Y)$
and $\bA + \bF = (A_1+F_1, \dots, A_m + F_m)$.
Consequently, $\Lambda_{\infty,q}(\bA)$ is convex.
\end{theorem}

Note that even though $\Lambda_{p,q}(\bA) \ne \emptyset$ for every positive integer $p$,
it is possible that $\Lambda_{\infty,q}(\bA) = \emptyset$; see
\cite[Example 4.7]{LP2}. In any event,
$\Lambda_{\infty,q}(\bA)$ can be constructed by the joint $q$-matricial range when $p_0 =1$ in Theorem \ref{thm33}.

\begin{corollary}\label{cor33}
Let $\bA  = (A_1, \dots, A_m) \in \cS(\cH)^m$ be an $m$-tuple of self-adjoint operators in $\cB(\cH)$. Then
\begin{eqnarray*}
\Lambda_{\infty,q}(\bA)
= \bigcap \{ W(q: Y^*\bA Y): Y \in \cV^\bot\}
= \bigcap \{ W(q:(\bA + \bF)): \bF \in \cS_\cF(\cH)^m\}.
\end{eqnarray*}
\end{corollary}

We divide the proof of Theorem \ref{thm33} into several lemmas, which are of independent interest.

\begin{lemma}\label{lem34}
Let $\bA = (A_1,\dots, A_m) \in \cS(\cH)^m$ be an $m$-tuple of self-adjoint operators in $\cB(\cH)$. Then
$$ \bigcap_{Y \in \cV^\bot} \Lambda_{p,q}(Y^*\bA Y)
\subseteq \bigcap_{Y \in \cV^\bot} W(q:Y^*\bA Y)
\subseteq \Lambda_{\infty,q}(\bA).$$
\end{lemma}

\begin{proof}
The first inclusion is clear. For the second inclusion, 
given $\bC=(C_1,...,C_m)\in \bigcap_{Y \in \cV^\bot} W(q: Y^*\bA Y)$, 
we claim that there exist an infinite sequence of operators $\{X_r\}_{r=1}^\infty \subseteq  \cV_q$
with $X_r:\cK_r \to \cH$ for some $q$-dimensional subspace $\cK_r$ of $\cH$,
such that any two distinct subspaces $\cK_r$  and $\cK_s$ are orthogonal and 
$$
X_r^* X_s = \begin{cases}
I_q & r = s, \\
0_q & r \ne s,
\end{cases}
\quad\hbox{and}\quad
X_r^* A_j X_s = \begin{cases}
C_j & r = s, \\
0_q & r \ne s.
\end{cases}$$
Once the claim holds, since $\{\cK_r\}_{r=1}^\infty$ is an infinite sequence of distinct orthogonal $q$-dimensional subspaces of $\cH$,
one can extend $\{X_r\}_{r=1}^\infty$ to an unitary operator $U:\cH \to \cH$ such that $U|_{\cK_r} = X_r$ for all $r$.
Then the operator matrix of $U^*A_j U$ with respect to the decomposition $\cH = \cK_1 \oplus \cK_2 \oplus \cK_3 \oplus \cdots$ has the form
$$\begin{bmatrix}
C_j & 0 & 0 & \cdots & \cr
0   &  C_j & 0 & \cdots  \cr
0 & 0 &  C_j  & \cdots \cr
\vdots & \vdots & \vdots & \ddots 
\end{bmatrix}.$$
Thus, $\bC \in \Lambda_{\infty,q}(\bA)$.
Now it remains to prove the claim,
which will be done by induction.

\medskip
Assume $\bC=(C_1,...,C_m)\in \bigcap_{Y \in \cV^\bot} W(q: Y^*\bA Y)$.
Then $\bC \in W(q:\bA) = \Lambda_{1,q}(\bA)$ and
there exists $X_1 \in \cV_{q}$ such that $X_1: \cK_1 \to \cH$ with $X_1^*X_1 = I_{\cK_1}$
for some $q$-dimensional subspace $\cK_1$ of $\cH$
so that $X_1^* A_j X_1 =  C_j$ for $j =1,\dots,m$.
The claim holds for $\{X_1\}$. 

Assume the operators $\{X_1,\dots,X_n\}$ already satisfy the claim.
Then $\cK_r$ and $\cK_s$ are orthogonal for all $1\le r < s \le n$.
Since $X_r^*X_s = 0_q$ for $1\le r < s\le n$,
$X_r(\cK_r)$ is orthogonal to $X_r(\cK_s)$ for $r \ne s$.
Then one can extend $X_1,\dots,X_n$ to an unitary operator $U: \cH \to \cH$ 
such that $U|_{\cK_r} = X_r$ for $1\le r \le n$,
and the operator matrix of $U^*A_jU$ with respect to the decomposition $\cH = \left(\bigoplus_{r=1}^n \cK_r\right) \oplus \left(\bigoplus_{r=1}^n \cK_r\right)^\bot$
has the form
$$\begin{bmatrix}
I_n \otimes C_j & *  \cr
 * & *
\end{bmatrix}.$$
Let $\cL$ be the subspace spanned by
$$\left\{ \left(\oplus_{r=1}^n \cK_r\right) ,\ U^*A_1U \left(\oplus_{r=1}^n \cK_r\right),\ \dots,\ U^*A_mU\left(\oplus_{r=1}^n \cK_r\right) \right\}.$$
Then $\cL$ has dimension at most $qr(m+1)$. Take $Y = U|_{\cL^\bot}$. Then $Y\in \cV^\bot$.
By the above assumption, $\bC \in W(q:Y^* \bA Y)$ and there exists
$X: \cK_{n+1} \to \cL^\bot$ with $X^*X = I_{\cK_{n+1}}$
for some $q$-dimensional subspace $\cK_n$ of $\cL^\bot$
such that 
$$X^*(Y^*A_j Y) X = C_j,\quad\hbox{$j =1,\dots,m$}.$$
Define $X_{n+1} = YX: \cK_{n+1} \to \cH$. Clearly, $X_{n+1}^*X_{n+1} = I_q$
and $X_{n+1}^*A_j X_{n+1} = C_j$ for all $j =1,\dots,m$.
Now fixed $1\le r \le n$.  For any $u\in \cK_{n+1}$  and $v \in \cK_{r}$, 
$Xu \in \cL^\bot$ and $U^*A_j U v \in \cL$. Then
$$\langle  u,  X_{n+1}^*X_r v\rangle
= \langle  X_{n+1}u, X_r v\rangle
= \langle  YXu, U v\rangle
= \langle  UXu, U v\rangle
= \langle  Xu,  v\rangle = 0$$
and
$$\langle  u,  X_{n+1}^*A_j X_r v\rangle
= \langle  X_{n+1} u,   A_j X_r v\rangle
= \langle  UX u,   A_j U v\rangle
= \langle  X u,   U^*A_j U v\rangle
= 0.$$
Thus, $X_{n+1} X_r = 0$ and $X_{n+1}A_j X_r = 0$ for all $1\le j \le m$ and $1\le r \le n$.
Thus, the operators $\{X_1,\dots,X_{n+1}\}$ satisfy the claim.
By induction, the claim holds in general.
\end{proof}

\begin{lemma}\label{lem33}
Let $\bA = (A_1,\dots, A_m) \in \cB(\cH)^m$ be an $m$-tuple of self-adjoint operators in $\cB(\cH)$.
\begin{enumerate}
\item[\rm (a)] For any $\bF\in \cS_\cF(\cH)^m$, there exists $Y \in \cV^\bot$ such that
$\Lambda_{p,q}(Y^* \bA Y  ) \subseteq \Lambda_{p,q}(\bA + \bF  )$.
\item[\rm (b)] For any $Y \in \cV^\bot$, there exists $\bF\in \cS_\cF(\cH)^m$ such that
$\Lambda_{p,q}(\bA + \bF  ) \subseteq \Lambda_{p,q}(Y^* \bA Y  )$.
\end{enumerate}
Consequently,
$$\bigcap_{\bF \in \cS_\cF(\cH)^m} \Lambda_{p,q}(\bA + \bF  )
=\bigcap_{ Y \in \cV^\bot} \Lambda_{p,q}(Y^* \bA Y ).$$
\end{lemma}

\begin{proof}
(a) For any $\bF \in \cS_\cF(\cH)^m$, there exists $Y \in \cV^\bot$ such that $Y^*\bF Y   = \b0 = (0,\dots, 0)$. Then
$$\Lambda_{p,q}(Y^*\bA Y)
= \Lambda_{p,q}(Y^*(\bA+\bF) Y)
\subseteq \Lambda_{p,q}(\bA+\bF).$$

\medskip\noindent
(b) Suppose $Y\in \cV^\bot_r$ for some $r$.
Then $Y: \cL_1^\bot \to \cH$ such that $Y^*Y = I_{\cL_1^\bot}$
for some $r$-dimensional subspace $\cL_1$ of $\cH$.
Since $\cL^\bot_1$ is infinite dimensional, by Theorem~\ref{main2},
the joint rank $pq$-numerical range
$\Lambda_{pq}(Y^*\bA Y)$ is non-empty.
Take an element $(b_1,\dots,b_m) \in \Lambda_{pq}(Y^* \bA Y)$.
Then there exists a $pq$-dimensional subspace $\cK_1$ of $\cL_1^\bot$ and $X: \cK_1 \to \cL_1^\bot$
with $X^*X = I_{\cK_1}$ such that $X^*(Y^*A_j Y)X = b_j I_{pq}$ for $j =1,\dots,m$.
Extend the operator $YX: \cK_1 \to \cH$ to an unitary operator $U: \cH \to \cH$ so that $U|_{\cK_1} = YX$.
Let $\cL_2$ be the subspace spanned by 
$$\{\cL_1, \cK_1,  U^*A_1U(\cK_1), \dots, U^*A_m U (\cK_1) \}.$$
Then $\cL_2$ has dimension at most $pqm+r$. Set $W = U|_{\cL_2^\bot}$.
Then the operator matrix of $U^*A_jU$ with respect to the decomposition $\cH = \cK_1 \oplus \cL_2^\bot\oplus (\cK_1 \oplus \cL_2^\bot)^\bot$ has the form
$$
\begin{bmatrix}
b_j I_{\cK_1} & 0 & * \cr
0 & W^*A_j W & * \cr
* & * & *
\end{bmatrix}.$$
Let $B_j$ be the self-adjoint operator such that the operator matrix of $U^*B_jU$ with respect to the same decomposition  $\cH = \cK_1 \oplus \cL_2^\bot\oplus (\cK_1 \oplus \cL_2^\bot)^\bot$ has the form
$$\begin{bmatrix}
b_j I_{\cK_1} & 0 & 0 \cr
0 & W^*A_j W & 0 \cr
0 & 0 & b_j I_{(\cK_1 \oplus \cL_2^\bot)^\bot}
\end{bmatrix}\quad\hbox{and}\quad
F_j = B_j - A_j.$$ 
Notice that $\cK_1\oplus (\cK_1 \oplus \cL_2^\bot)^\bot = \cL_2$ is finite dimensional and
$U^*F_j U$ has the form
$\begin{bmatrix}
0 & 0 & * \cr
0 & 0 & * \cr
* & * & *
\end{bmatrix}$.
Thus, $F_j$ is a finite rank operator.
Now 
denote $\bF = (F_1,\dots,F_m)\in \cS_\cF(\cH)^m$ and suppose  
$$\bC = (C_1,...,C_m)\in \Lambda_{p,q}(\bA + \bF) = \Lambda_{p,q}(\bB) = \Lambda_{p,q}(U^*\bB U).$$
Then there exists $Z: \cK\to \cH$ with $Z^*Z = I_{\cK}$
for some $pq$-dimensional subspace $\cK$ of $\cH$
such that
$Z^*(U^*B_j U) Z = I_p \otimes C_j$ for $ j =1,\dots,m$.
Write $Z = \begin{bmatrix} Z_1 \\  Z_2\\ Z_3 \\  \end{bmatrix}$
according to the same decomposition $\cH = \cK_1 \oplus \cL_2^\bot\oplus (\cK_1 \oplus \cL_2^\bot)^\bot$. Then
$$b_j Z_1^*Z_1 + Z_2^* (W^* A_jW) Z_2 + b_j Z_3^*Z_3 = Z^*(U^*B_jU) Z = I_p \otimes C_j.$$
Since $\dim \cK = pq=\dim \cK_1$, one can always find an operator $\hat Z_1$ such that
$\hat Z_1^*\hat Z_1 = Z_1^*Z_1  + Z_3^*Z_3$. Define
$\hat Z: \cK \to \cK_1 \oplus \cL_2^\bot$ by $\hat Z = \begin{bmatrix} \hat Z_1 \\ Z_2 \end{bmatrix}$
with respect to the decomposition $ \cK_1 \oplus \cL_2^\bot$.
Then $\hat Z^*\hat Z = I_{\cK}$ and hence $\hat Z \in \cV_{pq}$.
Furthermore,
\begin{eqnarray*}
\hat Z^*
\begin{bmatrix}
b_j I_{\cK_1} & 0  \cr
0 & W^*A_j W 
\end{bmatrix} \hat Z 
&=& b_j \hat Z_1^* \hat Z_1 + Z_2^*(W^* A_jW) Z_2 \\[1mm]
&=& b_j (  Z_1^*Z_1 + Z_3^*Z_3)  + Z_2^*(W^* A_jW) Z_2
= Z^*(U^*B_jU) Z = I_p \otimes C_j.
\end{eqnarray*}
Recall that $\cK_1 \subseteq \cL_1^\bot$ and $\cL_2^\bot \subseteq \cL_1^\bot$, and hence $\cK_1 \oplus \cL_2^\bot \subseteq \cL_1^\bot$.  Thus,
the operator $b_j I_{\cK_1} \oplus W^*A_j W$ is a compression of $Y^*A_j Y$ on $\cK_1 \oplus \cL_2^\bot$.
Thus, 
$$\bC \in \Lambda_{p,q}\left(
b_1 I_{\cK_1} \oplus W^*A_1 W, \dots,
b_m I_{\cK_1} \oplus W^*A_m W \right) \subseteq
\Lambda_{p,q}(Y^*\bA Y).$$
Hence, the proof is complete.
\end{proof}

\noindent
{\it Proof of Theorem \ref{thm33}.}
The last equality follows by Lemma \ref{lem33}. Now
by Proposition \ref{prop31}, Proposition \ref{thm27} and Lemma \ref{lem34},
\begin{eqnarray*}
\Lambda_{\infty,q}(\bA)
 &\subseteq &  \bigcap_{p\ge 1} S_{p,q}(\bA)
\subseteq  \bigcap_{p\ge 1} \Lambda_{p,q}(\bA)
\subseteq \bigcap_{p\ge 1} \Lambda_{p_0+pq,q} (\bA) \cr
&\subseteq& \bigcap_{p\ge 1} \left( \bigcap_{Y\in \cV^\bot_p} \Lambda_{p_0,q}(Y^*\bA Y)\right)
= \bigcap \{ \Lambda_{p_0,q}(Y^*\bA Y): Y \in \cV^\bot\} \subseteq \Lambda_{\infty,q}(\bA).
\end{eqnarray*}
Thus, the result follows. Note that $S_{p,q}(\bA)$ is convex for all positive integer $p$. Hence the last assertion follows.
\qed

The following fact can be easily deduced from Theorem \ref{thm33}.
\begin{corollary}\label{cor36}
Let $\bA = (A_1,\dots, A_m) \in \cS(\cH)^m$ be an $m$-tuple of self-adjoint operators in $\cB(\cH)$.
For any $Y_0 \in \cV^\bot$ and $\bF_0 \in \cS_\cF(\cH)^m$,
$$
\Lambda_{\infty,q}(\bA)
= \Lambda_{\infty,q}(Y_0^*\bA Y_0)
= \Lambda_{\infty,q}(\bA+ \bF_0).
$$
\end{corollary}

\medskip
Similar to Section 2, all the results in this section still hold if one replaces
$\cS(\cH)$ and $\cS_\cF(\cH)$ by $\cB(\cH)$ and $\cF(\cH)$ respectively.

\section{Essential Matricial Range}

In this section, we will continue to assume that
$\cH$ is of infinite dimension. Define
the joint essential $(p,q)$-matricial range of $\bA=(A_1,\dots,A_m)$ by
$$
\Lambda^{ess}_{p,q}(\bA) =
\bigcap \{ \cl(\Lambda_{p,q}(A_1+F_1, \dots, A_m+F_m)): \ F_1, \dots, F_m\in \cB(\cH)  \hbox{ are finite rank operators} \}.
$$
We will show that
$\Lambda^{ess}_{p,q}(\bA)$
is the same as  the essential $q$-matricial range
$$ W_{ess}(q:\bA) =
\bigcap\{ \cl\left(W(q:(A_1 + F_1, \dots, A_m + F_m))\right) : \ F_1, \dots, F_m\in \cB(\cH)
 \hbox{ are finite rank operators} \},
$$
which is a  non-empty compact convex set. Consequently, the sets
$\Lambda_{p,q}^{ess}(\bA)$ are the same for all positive integer $p$.
Note that the above definitions $\Lambda_{p,q}^{ess}(\bA)$ and  $W_{ess}(q:\bA)$
are slightly different from those given in the introduction.
We will show that in the definitions, one can replace
the $m$-tuple $(F_1, \dots, F_m)$ of finite rank operators
by $m$-tuples of compact operators after we obtain
many equivalent representations of the sets
$\Lambda_{p,q}^{ess}(\bA)$ and $W_{ess}(q:\bA)$
defined above.

\medskip
Observe that $W_{ess}(1: \bA) = W_{ess}(\bA)$,  the joint essential numerical range.
It was shown in \cite{LP1} that $W_{ess}(\bA)$ is always convex,
and one can use $m$-tuples of finite rank operators or compact operators
$(F_1, \dots, F_m)$ in the definition of $W_{ess}(\bA)$.

\begin{theorem}\label{main3}
Let $\bA\in \cS(\cH)^m$ be an $m$-tuple of self-adjoint operators in $\cB(\cH)$,
and let $p_0,p,q$ be positive integers. Then
$$\Lambda_{p,q}^{ess}(\bA) =  W_{ess}(q:\bA)$$
is a compact convex set containing $(a_1I_q, \dots, a_m I_q)$
for all $(a_1, \dots, a_m) \in  W_{ess}(\bA)$.
Consequently, the set $\Lambda_{p,q}^{ess}(\bA)$ is independent of the choice of $p$.
Moreover, if
$\tilde S_{p,q}(\bA)$ is the set
of star-centers of $\cl (\Lambda_{p,q}(\bA))$, then
$$\Lambda_{p,q}^{ess}(\bA) = \bigcap_{r\ge 1}\tilde S_{r,q}(\bA)=\bigcap_{r\ge 1} \cl \left(\Lambda_{r,q}\left(\bA\right)\right)
=\bigcap \{ \cl(\Lambda_{p_0,q}(Y^*\bA Y)): Y \in \cV^\bot\}.
$$
\end{theorem}

Before we present the proof of Theorem \ref{main3}, we begin with a general observation.
Let $(V,\left\|\cdot\right\|)$ be a normed space. For any star-shaped set $S\subseteq V$,
denote by $S^c$ the set of all star-centers of $S$. The following result is known. We
include a short proof here for completeness.

\begin{lemma}\label{star_center}
Let $S\subseteq V$ be star-shaped. Then
$$\cl(S^c)\subseteq (\cl(S))^c.$$
\end{lemma}
\begin{proof}
It suffices to show that $S^c\subseteq (\cl(S))^c$ as the set of all star-centers of
closed star-shaped set is closed. Let $c\in S^c$. Suppose that $\eta\in \cl(S)$. Then
there is a sequence $\{\eta_r\}\subseteq S$ converging to $\eta$. Note that
$tc+(1-t)\eta_r\in S$ for any $0\leq t\leq 1$ and positive integer $r$.
As $\{tc+(1-t)\eta_r\}\subseteq S$ converging to $tc+(1-t)\eta$,
we have $tc+(1-t)\eta\in\cl(S)$. Hence $c\in(\cl(S))^c$.
\end{proof}

\begin{lemma}\label{lem42}
Let $\bA=(A_1,\dots,A_m)\in \cS(\cH)^m$ be an $m$-tuple of self-adjoint operators in $\cB(\cH)$,
and let $p_0,p,q$ be positive integers. Then
$$
\bigcap \{ \cl(\Lambda_{p_0,q}(Y^*\bA Y)): Y \in \cV^\bot\}\subseteq
\bigcap \{ \cl(W(q:Y^*\bA Y)): Y \in \cV^\bot\}\subseteq \cl(\Lambda_{p,q}(\bA)).$$
\end{lemma}
\begin{proof}
The first inclusion is trivial. Let $\bC=(C_1,\dots,C_m)\in \bigcap \{ \cl(W(q:Y^*\bA Y)): Y \in \cV^\bot\}$.
Fix a positive integer $n$.
As $\bC\in \cl(W(q:\bA ))$, there exists $X_1\in \cV_q$ such that $X_1:\cK_1\to\cH$ with $X_1^*X_1=I_{\cK_1}$ so that $X_1^*A_j X_1=B^{(1)}_j$ and $\|B^{(1)}_j-C_j \|\leq \frac{1}{n}$, $j=1,...,m$. 
By a similar inductive argument in Lemma~\ref{lem34},
one can construct two infinite sequences $\{X_r\}_{r = 1}^\infty$
and $\{\cK_r\}_{r=1}^\infty$ with $X_r: \cK_r \to \cH$
for some $q$-dimensional subspace $\cK_r$ such that
any two distinct subspaces $\cK_r$ and $\cK_s$ are orthogonal,
$$
X_r^* X_s = \begin{cases}
I_q & r = s, \\
0_q & r \ne s,
\end{cases}
\quad\hbox{and}\quad
X_r^* A_j X_s = \begin{cases}
B_{j}^{(r)} & r = s, \\
0_q & r \ne s,
\end{cases}
\quad\hbox{with}\quad
\left\|B_j^{(r)} - C_j \right\| \le \frac{1}{n},$$
for $j = 1,\dots,m$.
Take $d \ge (p-1)(q^2m+1)+1$ and set $X = \oplus_{r=1}^d X_r: \oplus_{r=1}^d \cK_r \to \cH$.
Then $X^*X  = I_{\oplus_{r=1}^d \cK_r}$ and 
$$X^*A_j X = B_j^{(1)} \oplus B_j^{(2)} \oplus \cdots \oplus B_j^{(d)}
\quad\hbox{and}\quad
\left\|B_j^{(r)} - C_j \right\| \le \frac{1}{n} \quad\hbox{for } j = 1,\dots,m \hbox{ and }r = 1,\dots,d.$$
Denote $\bB_r=(B^{(r)}_1,\dots, B^{(r)}_m)$ for $r=1,\dots,d$. 
Identifying $\bB_1,\dots,\bB_d$ as $d$ points in $\mathbb{R}^{q^2m}$, then
by Tverberg's Theorem (see~\cite{Tverberg}), one can partition $\{\bB_r:r=1,\dots,d\}$ into $p$ sets 
$$\cB_1 = \{\bB_r: r \in \mathcal I_1\},\quad
\cB_2 = \{\bB_r: r \in \mathcal I_2\}, 
\quad \dots \quad
\cB_p = \{\bB_r: r \in \mathcal I_p\}$$
such that $\co( \mathcal B_1)\cap \cdots\cap \co( \mathcal B_p)\neq \emptyset$.
Pick $\bC^{(n)}=(C_1^{(n)},\dots,C_m^{(n)}) \in \co( \mathcal B_1)\cap \cdots\cap \co( \mathcal B_p)$.
Then 
$\bC^{(n)} \in \Lambda_{q} \left( \oplus_{r \in \mathcal I_\ell} \bB_j \right)$
for $\ell = 1,\dots,p$ and hence
$$
\bC^{(n)} \in \Lambda_{p,q} \left( \oplus _{t=1}^p \oplus_{r \in \mathcal I_\ell} \bB_j \right)
= \Lambda_{p,q}\left( \oplus _{t=1}^d \bB_j \right) = \Lambda_{p,q}(X^* \bA X) \subseteq \Lambda_{p,q}(\bA).
$$
Now as $C_j^{(n)}$ is a convex combination of $B_j^{(r)}$'s, $\| B_j^{(r)} - C_j \| \le \frac{1}{n}$ implies $\| C_j^{(n)} - C_j \| \le \frac{1}{n}$.
We have $\bC^{(n)}\in \Lambda_{p,q}(\bA)$ and $\left\|C_j^{(n)}-C_j\right\|<\frac{1}{n}$. Therefore, 
there exists a sequence $\{\bC^{(n)}\}_{n=1}^\infty\subseteq\Lambda_{p,q}(\bA)$ 
converging to $\bC$. Hence, $\bC\in \cl(\Lambda_{p,q}(\bA))$.
\end{proof}

\noindent
{\it Proof of Theorem \ref{main3}.} We prove the second assertion first.
First by Lemma~\ref{lem33}, one can obtain
$$\Lambda^{ess}_{p,q}(\bA)
= \bigcap \{\cl\left(\Lambda_{p,q}(\bA + \bF  )\right): \bF \in \cS_\cF(\cH)^m\}
=\bigcap \{\cl\left(\Lambda_{p,q}(Y^* \bA Y )\right): Y \in \cV^\bot\}.$$
Next by Corollary~\ref{coro25} and Lemma~\ref{star_center},
$$\bigcap_{r\ge 1} \tilde S_{r,q}(\bA)
\subseteq  \bigcap_{r\ge 1} \cl(\Lambda_{r,q}(\bA))
\subseteq \bigcap_{r\ge 1} \cl(\Lambda_{r(q^2(m+1)+1),q} (\bA))
\subseteq \bigcap_{r\ge 1}  \cl(S_{r,q}(\bA))
\subseteq \bigcap_{r\ge 1}  \tilde S_{r,q}(\bA).
$$
Now by Proposition \ref{thm27} and Lemma~\ref{lem42}, we have
\begin{multline*}
\hspace{1cm}
\bigcap_{r\ge 1} \cl(\Lambda_{r,q}(\bA))
\subseteq \bigcap_{r\ge 1} \cl(\Lambda_{p_0+rq,q} (\bA))
\subseteq \bigcap_{r\ge 1} \left( \bigcap_{Y\in \cV^\bot_r} \cl(\Lambda_{p_0,q}(Y^*\bA Y))\right)\\
= \bigcap \{ \cl(\Lambda_{p_0,q}(Y^*\bA Y)): Y \in \cV^\bot\}
\subseteq \bigcap_{r\ge 1} \cl\left(\Lambda_{r,q}(\bA)\right). \hspace{.7cm}
\end{multline*}
As $\tilde S_{r,q}(\bA)$ is a compact convex set for all positive integers $r$, its intersection is also compact and convex. Thus, we see that the chain of set equality holds in the last assertion,
and that $\Lambda_{p,q}^{ess}(\bA)$ is convex.

Note that $\Lambda_{p,q}^{ess}(\bA) =\bigcap_{r\ge 1}\tilde S_{r,q}(\bA)$,
which is independent on $p$. We see that
$$\Lambda_{p,q}^{ess}(\bA) = \Lambda_{1,q}^{ess}(\bA) = W_{ess}(q:\bA).$$
Finally, let $(a_1, \dots, a_m) \in W_{ess}(1,\bA)$. Then
$(a_1I_q, \dots, a_m I_q) \in \Lambda_{r,q}(\bA)$
for each positive integer $r$, and
$(a_1I_q, \dots, a_m I_q) \in \Lambda_{r,q}^{ess}(\bA)$.
\qed

By similar arguments as in Theorem~\ref{thm33} and Corollary~\ref{cor36}, we have the following result
giving other representations of $\Lambda_{p,q}^{ess}(\bA) = W_{ess}(q:\bA)$.

\begin{corollary} \label{cor3.9}
Let $\bA\in \cS(\cH)^m$ be an $m$-tuple of self-adjoint operators in $\cB(\cH)$. Let $p\geq 1$, $\bF_0\in \cS_\cF(\cH)^m$, and $Y_0\in\cV^\bot$. Then
$$\Lambda_{p,q}^{ess}(\bA) =   \Lambda_{p,q}^{ess}(Y_0^*\bA Y_0)
= \Lambda_{p,q}^{ess}(\bA + \bF_0) =
\bigcap \left\{\cl \left(\Lambda_{p,q}(\bA+\bF)\right):
	\bF\in\cS_\cF(\cH)^m\right\},$$
and is equal to
$$
W_{ess}(q:\bA) = W_{ess}(q:Y_0^*\bA Y_0)
	= W_{ess}(q:\bA+\bF_0)
= \bigcap \left\{\cl \left(W(q:\bA + \bF)\right):
	\bF\in\cS_\cF(\cH)^m \right\}.$$
\end{corollary}

Let $\cK(\cH)$
be the set of compact  operators
on the infinite dimensional Hilbert space $\cH$
and $\cS_\cK(\cH) = \cS(\cH) \cap \cK(\cH)$.
In the following, we show that the definitions of
$W_{ess}(q:\bA)$ and $\Lambda_{p,q}^{ess}(\bA)$ are the same
if we replace the $m$-tuples of
finite rank operators $(F_1, \dots, F_m)$ by
$m$-tuples of compact operators.

\begin{proposition} \label{thm4.1}
Let $\bA\in \cS(\cH)^m$ be an $m$-tuple of self-adjoint operators in $\cB(H)$. 
Then 
$$W_{ess}(q:\bA)= \bigcap \{ \cl (W(q:\bA + \bG): \bG \in \cS_\cK(\cH)^m\}.$$
\end{proposition}
\begin{proof}
Let
$$S =\bigcap \{ \cl (W(q:\bA + \bG): \bG \in \cS_\cK(\cH)^m\}.$$
Evidently, $S \subseteq W_{ess}(q:\bA)$ as $\cS_\cF(\cH) \subseteq \cS_\cK(\cH)$.
We focus on the reverse inclusion.
Suppose $\bC = (C_1, \dots, C_m) \in W_{ess}(q:\bA)$.
We will show that $\bC \in \cl(W(q: \bA+\bG))$
for any $\bG \in \cS_\cK(\cH)^m$.

Suppose $\bG = (G_1, \dots, G_m) \in \cS_\cK(\cH)^m$.
For any  given $\varepsilon > 0$, there exists $Y\in\cV^\perp$ such that
$\|Y^*G_iY\|<\varepsilon/2$ for all $1\le i\le m$.
By Corollary \ref{cor3.9},
we can find $X\in\cV_q$ such that $\|C_i-X^*Y^*A_iYX\|<\varepsilon/2$
for all $1\le i\le m$. Therefore, $\|C_i-X^*Y^*(A_i+G_i)YX\|<\varepsilon$ for
all $1\le i\le m$. Hence, $\bC \in \cl(W(q: \bA+\bG))$.
\end{proof}

\begin{proposition}
Suppose $\bA\in \cS(\cH)^m$ be an $m$-tuple of self-adjoint operators in $\cB(H)$. 
Then 
$$\Lambda^{ess}_{p,q}(\bA)=\bigcap \{ \cl (\Lambda_{p,q}(\bA + \bG): \bG \in \cS_\cK(\cH)^m\}.$$
\end{proposition}
\begin{proof}
Let
$$\tilde S = \bigcap \{ \cl (\Lambda_{p,q}(\bA + \bG): \bG \in \cS_\cK(\cH)^m\}.$$

Evidently, $\tilde S \subseteq \Lambda^{ess}_{p,q}(\bA)$
as $\cS_\cF(\cH) \subseteq \cS_\cK(\cH)$.
We consider the reverse inclusion.
By Corollary \ref{cor3.9} and Proposition \ref{thm4.1},
for any fixed $\bG \in \cS_\cK(\cH)^m$,
$$ \Lambda_{p,q}^{ess}(\bA)
= W_{ess}(q:\bA)
= W_{ess}(q:\bA+\bG)
= \Lambda_{p,q}^{ess}(\bA+\bG)
\subseteq \cl(\Lambda_{p,q}(\bA+\bG)).$$
Thus, $\Lambda_{p,q}^{ess}(\bA) \subseteq \bigcap \{ \cl (\Lambda_{p,q}(\bA + \bG): \bG \in \cS_\cK(\cH)^m\} = \tilde S$.
\end{proof}

\medskip
Clearly, by the above propositions, in Corollary \ref{cor3.9}
one may replace ``$\bF_0\in \cS_\cF(\cH)^m$''
by ``$\bF_0\in \cS_\cK(\cH)^m$''.
Again all the results in this section still hold if one replaces
$\cS(\cH)$, $\cS_\cF(\cH)$ and $\cS_\cK(\cH)$ by
$\cB(\cH)$, $\cF(\cH)$ and $\cK(\cH)$, respectively.

\section{Related Results}
One may use our techniques to obtain similar results in Section~2, with minor modification, on other types of generalized matricial ranges.
We list a few examples below.

\begin{enumerate}
\item The joint $(p,q)$-matricial range of an $m$-tuple of
$n\times n$
real (symmetric) matrices $\bA = (A_1, \dots, A_m) \in \bM_n(\IR)^m$ for $pq \le n$ defined by
$$\Lambda_{p,q}(\bA) = \{(B_1 \dots, B_m): X^t A_j X = I_p \otimes B_j
\hbox{ for an } n\times pq \hbox{ real matrix satisfying } X^tX = I_{pq}\}.$$

\item The joint $(p,q)$-congruence matricial range of an $m$-tuple of
$n\times n$ complex (symmetric or skew-symmetric) matrices $\bA = (A_1, \dots, A_m)\in \bM_n(\IC)^m$
for $pq \le n$ defined by
$$\Lambda^{con}_{p,q}(\bA) = \{(B_1 \dots, B_m): X^t A_j X = I_p \otimes B_j
\hbox{ for an } n\times pq \hbox{ complex matrix satisfying } X^*X = I_{pq}\}.$$

\end{enumerate}

One may also extend the results on the generalized matricial ranges mentioned above
to infinite dimensional real or complex Hilbert spaces. Furthermore, one can define similar
essential matricial ranges accordingly and consider the problems discussed in Section~4.

\section*{Acknowledgment}
The authors would like to thank the referee for her/his helpful suggestions.
Li is an honorary professor of the Shanghai University,
and an affiliate member of the Institute for Quantum Computing, University of
Waterloo; his research was supported by
the USA NSF DMS 1331021, the Simons Foundation
Grant 351047, and NNSF of China Grant 11571220.
Research of Sze was supported by a PolyU central research grant G-YBNS and a HK RGC grant PolyU 502512.
The HK RGC grant also supported the post-doctoral fellowship of Lau at the Hong Kong Polytechnic University.
Part of the work was done while Li and Poon were visiting
Institute for Quantum Computing at the University of Waterloo.
They gratefully acknowledged the support and kind hospitality of the Institute.



\begin{thebibliography}{99}
\addcontentsline{toc}{chapter}{Bibliography}
\bibitem{AL}
T. Ando and C.K. Li,
Special issue: the numerical range and numerical radius,
Linear and Multilinear Algebra 37 (1994), 1-238.


\bibitem{Arveson}
W. Arveson, Subalgebras of $C^*$-algebras II, Acta Math. 128 (1972), 271-308.

\bibitem{AYP}
Y.H. Au-Yeung and Y.T. Poon,
A remark on the convexity and positive definiteness concerning Hermitian matrices, Southeast Bull. Math. 3 (1979), 85-92.


\bibitem{AYT}
Y.H. Au-Yeung and N. K. Tsing,
An extension of the Hausdorff-Toeplitz theorem on the numerical range,
Proc. Amer. Math. Soc. 89 (1983), 215-218.



\bibitem{Chien2014}
M.T. Chien and H. Nakazato,
Reduction of joint c-numerical ranges,
Applied Mathematics and Computation 232 (2014), 178-182.


\bibitem{CKZ} M.D. Choi, D.W. Kribs and K. Zyczkowski,
Quantum error correcting codes from the compression formalism,
Rep. Math. Phys. 58 (2006), 77-91.

\bibitem{CLP} M.D. Choi, C.K. Li and Y.T. Poon,
Some convexity features associated with unitary orbits,
Canad. J. Math. 55 (2003), 91-111.

\bibitem{Fan}
M.K.H. Fan and A.L. Tits,
m-form numerical range and the computation of the structured singular value,
IEEE Trans. Automat. Control 33 (1988), 284-289.

\bibitem{Farenick} D.R. Farenick,
Matricial extensions of the numerical range: A brief survey,
Linear and Multilinear Algebra 34 (1993), Issue 3-4.

\bibitem{Halmos} P.R. Halmos,
A Hilbert Space Problem Book (2nd edition),
Springer, New York, 1982.

\bibitem{Hau}
F. Hausdorff,
Das Wertvorrat einer Bilinearform,
Math. Zeit. 3 (1919), 314-316.

\bibitem{HJ}
R.A. Horn and C.R. Johnson,
Topics in matrix analysis,
Cambridge Univ. Press, Cambridge, 1991.

\bibitem{GR}
K.E. Gustafson and D.K.M. Rao,
Numerical range: The field of values of linear operators and matrices,
Springer, New York, 1996.

\bibitem{LP0} C.K. Li and Y.T. Poon,
Convexity of the joint numerical range,
SIAM J. Matrix Analysis Appl. 21 (1999), 668-678.

\bibitem{LP1} C.K. Li and Y.T. Poon,
The joint essential numerical range of operators: Convexity and related results,
Studia Math. 194 (2009), 91-104.

\bibitem{LP2} C.K. Li and Y.T. Poon,
Generalized numerical ranges and quantum error correction,
J. Operator Theory 66 (2011), 335-351.

\bibitem{LPS0} C.K. Li, Y.T. Poon and N.S. Sze,
Condition for the higher rank numerical
range to be non-empty,
Linear and Multilinear Algebra 57 (2009), 365-368.


\bibitem{LPS} C.K. Li, Y.T. Poon and N.S. Sze,
Generalized interlacing inequalities,
Linear and Multilinear Algebra 60 (2012), 1245-1254.

\bibitem{LT} C.K. Li and N.K. Tsing,
On the $k$-th matrix numerical range,
Linear and Multilinear Algebra 28 (1991), 229-239.

\bibitem{Muller}
V. M\"{u}ller,
The joint essential numerical range, compact perturbations, and the Olsen problem,
Studia Math. 197 (2010), 275-290.

\bibitem{Paulsen}
V.I. Paulsen,
Preservation of essential matrix ranges by compact perturbations,
J. Operator Theory 8 (1982), 299-317.

\bibitem{Thompson} R.C. Thompson, Research problem: The matrix numerical range,
Linear and Multilinear Algebra
21 (1987), 321-323.

\bibitem{Toe}
O. Toeplitz,
Das algebraische Analogon zu einem Satze von Fejer,
Math. Zeit.
2 (1918), 187-197.

\bibitem{TW} S.H. Tso and P.Y. Wu,
Matricial ranges of quadratic operators,
Rocky Mountain J. Math.
29 (1999), 1139-1152.

\bibitem{Tverberg} H. Tverberg,
A generalization of Radon's theorem,
J. Lond. Math. Soc. 41 (1966), 123-128.

\end{thebibliography}
\end{document}